\documentclass{amsart}
\usepackage[margin=2cm]{geometry}
\usepackage{amssymb}
\usepackage{enumerate}   
\usepackage{graphicx}
\usepackage{tikz-cd} 
\usepackage{color}
\usepackage[colorlinks]{hyperref}

\theoremstyle{definition}
\newtheorem{theorem}{Theorem}[section]

\newtheorem{remark}[theorem]{Remark}
\newtheorem{lemma}[theorem]{Lemma}

\theoremstyle{remark}
\newtheorem{claim}{Claim}

\def\Tbb{\mathbb{T}}
\def\Zbb{\mathbb{Z}}
\def\Rbb{\mathbb{R}}
\def\Cbb{\mathbb{C}}

\newcommand{\G}{
	\mathrm{SL}_{2}(\mathbb{C})
}

\newcommand{\g}{
	\mathfrak{g}
}

\begin{document}
	
	\title{Adjoint Reidemeister torsions of two-bridge knots}
	
	\author{Seokbeom Yoon}
	\email{sbyoon15@kias.re.kr}
	%	
	
	%    General info
	%	 \subjclass[2020]{Primary 54C40, 14E20; Secondary 46E25, 20C20}
	
	% \date{January 1, 2001 and, in revised form, June 22, 2001.}
	
	% \dedicatory{This paper is dedicated to our advisors.}
	
	\keywords{Reidemeister torsion, adjoint representation, two-bridge knot, Riley polynomial, character variety}
	
	\begin{abstract} 
		We give an explicit formula for the adjoint Reidemeister torsion of two-bridge knots and prove that the adjoint Reidemeister torsion satisfies a certain type of vanishing identities.
	\end{abstract}
	
	\maketitle
	\tableofcontents

	\section{Introduction}
	Let $K$ be a knot in $S^3$ and $M$ be the knot exterior.
	For an irreducible representation $\rho:\pi_1(M)\rightarrow \G$ 
	the \emph{adjoint Reidemeister torsion} $\Tbb_\gamma(\rho)$ is defined, under some reasonable assumptions, as the sign-refined algebraic torsion of the cochain complex of $M$ with the coefficient $\mathfrak{sl}_{2}(\mathbb{C})
	$ twisted by the adjoint representation of $\rho$.
	Here $\gamma$ is a simple closed curve in $\partial M$ which has a role in specifying a basis of the twisted cohomology group $H^\ast(M;	\mathfrak{sl}_{2}(\mathbb{C})_\rho
	)$, see \cite{porti1997torsion, dubois2003torsion}.
	It is known that $\mathbb{T}_\gamma(\rho)$ is invariant under conjugation, so the notion of adjoint Reidemeister torsion is also well-defined for the character of $\rho$.
	
	The adjoint torsion is quite complicated to compute in general and its explicit formula is known only in a few examples, \cite{dubois2009non, tran2014twisted, tran2018twisted}.
	All successful computations so far first compute the adjoint twisted Alexander polynomial $\Delta_{K}^{\mathrm{Ad}\rho}(t)$ and then obtain the adjoint torsion from the formula of \cite{yamaguchi2008relationship}
	\[ \Tbb_\lambda(\rho) =- \lim_{t \rightarrow 1} \frac{\Delta_K^{\mathrm{Ad} \rho}(t)	}{t-1}, \quad \lambda : \textrm{the canonical longitude of $K$.}\]
	%	We refer to \cite{kitano1996twisted} for definition of $\Delta_K^{\mathrm{Ad}\rho}(t)$.
	It is computationally advantageous that the adjoint twisted Alexander polynomial is defined from an acyclic chain complex.
	However, the computation itself would be further complicated, as an indeterminant $t$ interferes in the Fox differential calculus.
	
	In this paper, we give an explicit formula for the adjoint Reidemeister torsion of two-bridge knots.
	Our computation uses a well-known observation in \cite{weil1964remarks} that relates the twisted cohomology group and the tangent space of the character variety.
	We stress that it is simple and effective, compared to the method using the adjoint twisted Alexander polynomial, as we do not require the Fox differential calculus and  we can directly obtain a relation of the adjoint torsion and the character variety.
	We here summarize our results briefly. 
	
	Let $K$ be a two-bridge knot given by the Schubert normal form $(p,q)$. Here $p>0$ and $q$ are relatively prime odd integers with $-p<q<p$. 
	The knot group of $K$ has a presentation
	\begin{equation*}
		\pi_1(M)= \langle g_1, g_2 \,|\, w g_1 w^{-1} g_2^{-1}  \rangle
	\end{equation*}
	where $w = g_1^{\epsilon_1}g_2^{\epsilon_2}\cdots g_1^{\epsilon_{p-2}}g_2^{\epsilon_{p-1}}$ and $\epsilon_i = (-1)^{\lfloor \frac{iq}{p} \rfloor}$. 
	Here $\lfloor x \rfloor$ means the
	greatest integer less than or equal to $x \in \Rbb$.
	Up to conjugation, an irreducible representation $\rho : \pi_1(M) \rightarrow \G$ is given by 
	\begin{equation*}
		\rho(g_1) = \begin{pmatrix} m & 1 \\ 0 & m^{-1} \end{pmatrix}, \ \rho(g_2) = \begin{pmatrix} m & 0 \\ -u & m^{-1} \end{pmatrix}
	\end{equation*}
	for some $m \neq 0$ and $u \neq 0$ satisfying the Riley polynomial of $K$, \cite{riley1984nonabelian}. We denote by $\chi_\rho$ the character of $\rho$ and let $\phi_g := g _{11} - (m-m^{-1})g_{12}$ for $g \in \pi_1(M)$ where $g_{ij} \in \Zbb[m^{\pm1},u]$ is the $(i,j)$-entry of $\rho(g)$. 
	Note that the Riley polynomial of $K$  is equal to $\phi_w$.
	\begin{theorem}[Theorem \ref{thm:main}] \label{thm:int1} 
		Suppose that $\chi_\rho$ is $\mu$-regular with $m \neq \pm1$, $\frac{\partial \phi_w}{\partial m} \neq0$, and $\frac{\partial \phi_w}{\partial u} \neq0$. Then we have
		\[ \Tbb_\mu(\chi_\rho)= \pm \, \frac{m^{\epsilon_k + 1}}{2(m^2-1)} \frac{w_{11}}{v_{11}\phi_v} \frac{\partial \phi_w}{\partial u}\]
		where $v=g_1^{\epsilon_1} g_2^{\epsilon_2} \cdots g_1^{\epsilon_{k-2}} g_2^{\epsilon_{k-1}}$ and $0 < k < p$ is a unique (odd) integer satisfying $k q \equiv \pm 1 \ (\mathrm{mod}\ 2p)$. Here the sign $\pm$ does not depend on the choice of $\chi_\rho$.
	\end{theorem}
	
	The adjoint torsion has intriguing and fruitful interactions with quantum field theory, see e.g. \cite{witten1989quantum, gukov2005three}.
	Recently, it is showed in \cite{benini2019rotating, gang2019precision}  that a certain sum of adjoint torsions realizes the so-called Witten index. 
	As a mathematical byproduct, it is conjectured that if a hyperbolic knot $K$ has the character variety $X_K$ of irreducible $\G$-representations consisting of 1-dimensional components (which is the case of hyperbolic two-bridge knots), then we have
	\begin{equation} \label{eqn:conj}
	\sum_{\chi_\rho \in \mathrm{tr}_\mu^{-1}(c)} \frac{1}{\mathbb{T}_\mu(\chi_\rho)} = 0 \quad \textrm{for generic } c \in \Cbb
	\end{equation}
	where $\mathrm{tr}_\mu : X_K \rightarrow \Cbb$ is the trace function of a meridian $\mu$. See \cite{gang2019adjoint} for details.
%	As  a consequence of Theorem \ref{thm:int1}, we prove that this conjecture holds for all hyperbolic two-bridge knots.
	\begin{theorem}[Theorem \ref{thm:main2}] \label{thm:int2} The equation \eqref{eqn:conj} holds for all hyperbolic two-bridge knots.
	\end{theorem}
	
	The paper is organized as follows. We review some backgrounds on the adjoint Reidemeister torsion in Section \ref{sec:basic}. We state our main results, Theorems \ref{thm:main} and \ref{thm:main2}, in  Section \ref{sec:twobridge} and give proofs of Theorems \ref{thm:main} and \ref{thm:main2} in Sections \ref{sec:proof1} and \ref{sec:proof2}, respectively.

	\subsection*{Acknowledgment}
	The author is supported by a KIAS Individual Grant (MG073801) at Korea Institute for Advanced Study.
	
	%%%%%%%%%%%%%%%%%%
	
	\section{Preliminaries: the adjoint Reidemeister torsion} \label{sec:basic}
	\subsection{Basic definitions} We briefly recall some basic definitions and known results that we need in the following sections. We mainly follow \cite{dubois2006non} and refer to \cite{porti1997torsion, turaev2002torsions, dubois2003torsion} for details.
	In what follows, we denote by $\g$ the Lie algebra of $\G$ with a standard basis
	\begin{equation} \label{eqn:basis}
	e_1= \begin{pmatrix}
	0 & 1 \\ 0	 & 0
	\end{pmatrix}
	,\
	e_2= \begin{pmatrix}
	1 & 0 \\ 0 & -1
	\end{pmatrix}
	,\
	e_3= \begin{pmatrix}
	0 & 0 \\ 1 & 0
	\end{pmatrix}
	\end{equation} and $\langle \cdot, \cdot\rangle_\g$ the Killing form of $\g$
	\begin{equation*}
		\left\langle \begin{pmatrix} b & a \\ c & -b \end{pmatrix}, \begin{pmatrix} b' & a' \\ c' & -b' \end{pmatrix}  \right\rangle_\g = 8  b b' +4 (a c'+ca').
	\end{equation*}

	Let $C_\ast=(0 \rightarrow C_n \rightarrow \cdots  \rightarrow C_0\rightarrow 0)$ be a chain complex of $\Cbb$-vector spaces with boundary maps $\partial_i : C_i \rightarrow C_{i-1}$. 
	For a given basis $c_\ast$ of $C_\ast$ and a given basis $h_\ast$ of the homology group $H_\ast(C_\ast)$, the \emph{algebraic torsion} $\mathrm{tor}(C_\ast,c_\ast, h_\ast)$ is defined 	as follows.
	For each $0\leq i \leq n$ let $b_i$ be any sequence of vectors in $C_i$ such that $\partial_i(b_i)$ is a basis of $\mathrm{Im}\,\partial_i$ and let $\widetilde{h}_i$ be any representative of $h_i$ in $C_i$. We then obtain a new basis  of $C_i$ by combining $\partial_{i+1}(b_{i+1})$, $\widetilde{h}_i$, and $b_i$ in order and let 
	\begin{equation}\label{eqn:torsion}
	\mathrm{tor}(C_\ast,c_\ast, h_\ast) :=\prod_{i=0}^n  \left[(\partial_{i+1}(b_{i+1}),\widetilde{h}_i, b_i)/c_i\right]^{(-1)^{i+1}} \in \Cbb^\ast.
	\end{equation}
	Here $[x/y]$ means the determinant of the transition matrix sending the basis $y$ to the other basis $x$.
	%	\[|C_\ast|=\sum_{j=0}^n \alpha_j \beta_j \quad \textrm{where } \alpha_j = \sum_{i=0}^j \dim C_i,\ \beta_j = \sum_{i=0}^j \dim H_i(C_\ast).\]

	Let $K$ be a knot in $S^3$ and $M$ be the knot exterior with a fixed triangulation. For an irreducible representation $\rho : \pi_1(M)\rightarrow \G$ we consider the twisted cochain complex
	\[C^\ast(M;\g_\rho):=\mathrm{Hom}_{\Zbb[\pi_1 M]} \left( C_\ast(\widetilde{M};\Zbb),\g\right)\]
	where $\widetilde{M}$ is the universal cover of $M$ with the induced triangulation and $\g$ is endowed with a $\Zbb[\pi_1 (M)]$-module structure via the adjoint representation of $\rho$.
	We fix an orientation of each cell in $M$ and denote the cells of $M$ by $c_1,\cdots,c_m$.
	Once we choose  a lift $\widetilde{c}_j$ of each $c_j$ to $\widetilde{M}$ $(1 \leq j\leq m)$, we denote by  $\mathbf{c}^\ast$ a basis of $C^\ast(M;\g_\rho)$ given as
	\begin{equation} \label{eqn:cbasis}
	\mathbf{c}^\ast = \left(c_1^1,c_1^2,c_1^3,\cdots,c_{j}^{1},c_{j}^{2},c_{j}^{3},\cdots,c_{m}^1,c_{m}^2,c_{m}^3\right)	
	\end{equation}
	where $c_j^i$ sends the cell $\widetilde{c}_j$ to $e_i$ and the other cells $\widetilde{c}_k$ $(1 \leq k \neq j \leq m)$ to 0.

	We assume that $\rho$ is \emph{$\mu$-regular} for a meridian $\mu$ of $K$, see \cite{porti1997torsion} for the definition of $\mu$-regularity. It follows that the twisted cohomology group $H^i(M;\g_\rho)$ has dimension 1 for $i=1,2$ and is trivial, otherwise.
	Furthermore, choosing a non-trivial element $P \in \g$ invariant under the adjoint action of $\rho(g)$ for all $g\in \pi_1(\partial M)$, we have isomorphisms $F_i : H^i(M;\g_\rho)\rightarrow \Cbb$ for $i=1,2$ (see  \cite{porti1997torsion, dubois2006non} for details) defined by
	\[F_1(v) =  \langle P, v(\widetilde{\mu}) \rangle_\g,\ F_2(v) =  \langle P, v(\widetilde{\partial M}) \rangle_\g\]
	where $\widetilde{\mu} \in C_1(\widetilde{M};\Zbb)$ and $\widetilde{\partial M}\in C_2(\widetilde{M};\Zbb)$ represent lifts of $\mu$ and $\partial M$ to $\widetilde{M}$, respectively, having the same base point.
	Here $\mu$ and $\partial M$ are regarded as to be coherently oriented, see \cite{dubois2006non} for details.
	Letting  $\mathbf{h}^i$ be a basis of $H^i(M;\g_\rho)$ satisfying $F_i(\mathbf{h}^i)=1$ for $i=1,2$, the \emph{adjoint Reidemeister torsion} $\Tbb_\mu(\rho)$ is defined as
	\begin{equation} \label{eqn:def}
	\Tbb_\mu(\rho)= \epsilon  \cdot \mathrm{tor} \left( C^\ast(M;\g_\rho),\mathbf{c}^\ast,\mathbf{h}^\ast \right) \in \Cbb^\ast.
	\end{equation}
	Here $\epsilon \in \{ \pm1\}$ is the sign determined by so-called Turaev's sign trick and the choice of homology orientation, see \cite{turaev2001introduction} for details.
	\begin{remark}
		It is known that the notion of $\mu$-regularity and the adjoint torsion are invariant under conjugation, so the adjoint Reidemeister torsion  is also well-defined for $\mu$-regular characters.
		Recall that the character of an irreducible representation $\rho$ determines 
		and is  determined by the conjugacy class of $\rho$, see \cite[Proposition 1.5.2]{culler1983varieties}.
	\end{remark}	
	
	%Also, 
	%he set of $\mu$-regular irreducible characters in $X_K$ is Zariski-open, see 

	\subsection{Infinitesimal deformation}
	We now fix a finite presentation of the knot group $\pi_1(M)$ of deficiency $1$
	\[\pi_1(M) = \langle g_1, \cdots, g_n \, | \, r_1,\cdots,r_{n-1} \rangle \]
	and let $Y$ be the corresponding 2-dimensional cell complex.
	Recall that $Y$ has one 0-cell $p$, $n$ 1-cells $g_1,\cdots,g_n$, and $n-1$ 2-cells $r_1,\cdots,r_{n-1}$.
	Since we may use $Y$ instead of the knot exterior $M$ to compute the adjoint torsion, we consider the twisted cochain complex of $Y$
	\[ 0 \rightarrow C^0(Y;\g_\rho) \overset{\delta^0}{\longrightarrow} C^1(Y;\g_\rho) \overset{\delta^1}{\longrightarrow} C^2(Y;\g_\rho) \rightarrow 0.\]
	Once we fix a lift  of the base point $p$ to the universal cover of $Y$, each cell of $Y$ admits a unique lift correspondingly. Denoting these lifts by using the usual symbol tilde $\widetilde{p}$, $\widetilde{g}_i$, and $\widetilde{r}_j$, it specifies  the basis $\mathbf{c}^\ast$ of $C^\ast(Y;\g_\rho)$ given as in the equation \eqref{eqn:cbasis}.
	With respect to the basis $\mathbf{c}^\ast$, it is known that
	\[\delta^0 = \begin{pmatrix}
	\Phi(g_1-1) \\
	\vdots\\
	\Phi(g_n-1)
	\end{pmatrix}, \quad \delta^1 = \begin{pmatrix}
	\Phi( \frac{\partial r_1}{\partial g_1}) & \cdots & \Phi( \frac{\partial r_1}{\partial g_n}) \\
	\vdots & \ddots & \vdots	\\
	\Phi(\frac{\partial r_{n-1}}{\partial g_1}) & \cdots & \Phi(\frac{\partial r_{n-1}}{\partial g_n})
	\end{pmatrix}
	\] where $\Phi : \Zbb[\pi_1(M)] \rightarrow M_{3,3}(\Cbb)$ is the $\Zbb$-linear extension of the adjoint representation of $\rho$ and $\partial r_j/\partial g_i$ is the Fox free differential.  Here $M_{3,3}(\Cbb)$ is the set of all 3-by-3 matrices. We refer to  \cite{kitano1996twisted, dubois2009non} for details.

	Let $F_n$ be the free group generated by $g_1,\cdots,g_n$ and suppose that we have a one-parameter family of representations $\rho_t : F_n \rightarrow \G$ (parameterized by $t\in\Rbb$) such that $\rho_{t_0}(g_1)=\rho(g_1), \cdots, \rho_{t_0}(g_n)=\rho(g_n)$ for some $t_0 \in \Rbb$.
	Assuming that the entries of $\rho_t(g_1),\cdots,\rho_t(g_n)$ are differentiable at $t=t_0$, we define $A_{\rho_t} \in C^1(Y;\g_\rho)$ by 
	\begin{equation} \label{eqn:A}
	A_{\rho_t} (\widetilde{g}_i) = \left. \frac{d}{dt}\right |_{t=t_0} \rho_t(g_i) \rho_{t_0}(g_i)^{-1} \quad \textrm{for } 1 \leq i \leq n.
	\end{equation}
	It follows from \cite{weil1964remarks} (see also \cite{sikora2012character}) that $\delta^1 (A_{\rho_t}) \in C^2(Y;\g_\rho)$ satisfies (and hence is determined by)
	\begin{equation}\label{eqn:r} (\delta^1 A_{\rho_t})(\widetilde{r}_j) = \left. \frac{d}{dt} \right |_{t=t_0} \rho_t(r_j) \quad \textrm{for } 1 \leq j \leq n-1.
	\end{equation}
	Note that we have $\rho_{t_0}(r_1)=\cdots=\rho_{t_0}(r_{n-1})=I$.

	\section{Adjoint Reidemeister torsion of two-bridge knots} \label{sec:twobridge}
	Let $K$ be a two-bridge knot given by the Schubert normal form $(p,q)$. Here $p>0$ and $q$ are relatively prime odd integers with $-p<q<p$. We refer to  \cite[Chapter 12]{burde2013knots} for details on two-bridge knot. 
	Let $M$ be the knot exterior of $K$.
	The knot group $\pi_1(M)$ of $K$ has a presentation
	\begin{equation} \label{eqn:prsn}
	\pi_1(M) = \langle g_1,g_2 \,|\, w g_1 w^{-1} g_2^{-1}  \rangle
	\end{equation}
	where $w = g_1^{\epsilon_1}g_2^{\epsilon_2}\cdots g_1^{\epsilon_{p-2}}g_2^{\epsilon_{p-1}}$ and $\epsilon_i = (-1)^{\lfloor \frac{iq}{p} \rfloor}$. 
	Here $\lfloor x \rfloor$ means the
	greatest integer less than or equal to $x \in \Rbb$.
	Up to conjugation, an irreducible representation $\rho : \pi_1(M) \rightarrow \G$ is given by 
	\begin{equation} \label{eqn:rep}
	\rho(g_1) = \begin{pmatrix} m & 1 \\ 0 & m^{-1} \end{pmatrix}, \ \rho(g_2) = \begin{pmatrix} m & 0 \\ -u & m^{-1} \end{pmatrix}
	\end{equation}
	for some $m \neq 0$ and $u \neq 0$ satisfying the \emph{Riley polynomial} of $K$, see the equation \eqref{eqn:riley} below.
	We denote by $\chi_\rho$ the character of $\rho$ and let
	$\phi_g := g_{11} - (m-m^{-1})g_{12}$ for $g \in \pi_1(M)$ where  $g_{ij} \in \Zbb[m^{\pm1},u]$ is the $(i,j)$-entry of $\rho(g)$. It is known that the Riley polynomial of $K$  is equal to
	\begin{equation} \label{eqn:riley}
	\phi_w = w_{11}-(m-m^{-1}) w_{12} \in \Zbb [m^{\pm1},u]
	\end{equation}	
	and the character variety $X_K$ of irreducible $\G$-representations is given by
	\[X_K = \left\{ (m,u) \in \Cbb^2 : \phi_w(m,u)=0,\ m\neq0,\ u \neq 0\right\}/_{\sim}\]
	where the quotient $\sim$ means that we identify $(m,u)$ and $(m^{-1},u)$. We refer to \cite{riley1984nonabelian} for details.
	
	\begin{theorem} \label{thm:main} Suppose that $\chi_\rho$ is $\mu$-regular with $m \neq \pm1$, $\frac{\partial \phi_w}{\partial m} \neq0$, and $\frac{\partial \phi_w}{\partial u} \neq0$. Then we have
		\[ \Tbb_\mu(\chi_\rho)=\pm \, \frac{m^{\epsilon_k+1}}{2(m^2-1)} \frac{w_{11}}{v_{11}\phi_v} \frac{\partial \phi_w}{\partial u}\]		
		where $v=g_1^{\epsilon_1}g_2^{\epsilon_2} \cdots g_1^{\epsilon_{k-2}} g_2^{\epsilon_{k-1}}$ and $0 < k < p$ is a unique (odd) integer satisfying $k q \equiv \pm 1 \ (\mathrm{mod}\ 2p)$. 
		Here the sign $\pm$ does not depend on the choice of $\chi_\rho$.
	\end{theorem}
	\begin{remark} The Riley polynomial $\phi_w$ has no repeated factor, see \cite[Lemma 3]{riley1984nonabelian}. It follows that the number of $\mu$-regular characters excluded in Theorem \ref{thm:main} is finite. In addition, the author does not know whether there exists an irreducible  $\mu$-regular character with either $\frac{\partial \phi_w}{\partial m} =0$ or $\frac{\partial \phi_w}{\partial u} =0$.
	\end{remark}	
	\begin{remark}  
		Othsuki and Takata \cite{ohtsuki2015kashaev} gave a diagrammatic formula for computing the adjoint torsion for a geometric representation of a hyperbolic two-bridge knot. Their formula can be directly applied to any irreducible $\mu$-regular character of a two-bridge knot with $m =\pm 1$.
	\end{remark}
	
	Let  $\mathrm{tr}_\mu : X_K \rightarrow \Cbb$ be the trace function of a meridian $\mu$, i.e., $\mathrm{tr}_\mu(\chi_\rho) = \mathrm{tr}\,\rho(\mu)$.
	It is known that non-$\mu$-regular characters in $X_K$ are contained in the set of zeros of the differential of $\mathrm{tr}_\mu:X_K \rightarrow \Cbb$, see
	\cite[Proposition 3.26]{porti1997torsion}.
	It follows that the set $\mathrm{tr}_\mu^{-1}(c)$ consists of $\mu$-regular characters for generic $c \in \Cbb$. As a consequence of Theorem \ref{thm:main}, we prove the following.
	
	\begin{theorem} \label{thm:main2} Suppose that $K$ is a hyperbolic two-bridge knot. Then we have
		\[\sum_{\chi_\rho \in \mathrm{tr}_\mu^{-1}(c)} \frac{1}{\mathbb{T}_\mu(\chi_\rho)} = 0\]		
		for generic $c \in \Cbb$.
	\end{theorem}
	\begin{remark} 
		It is proved in \cite{schubert1954numerische} that a two-bridge knot $K$ is hyperbolic if $q \neq \pm 1$ and is a torus knot of the type $(2,n)$, otherwise.
		For a non-hyperbolic two-bridge knot $(q=\pm1)$ we have $k=1$  and
		\[ \Tbb_\mu(\chi_\rho)=\frac{m^{q+1}\, w_{11} }{2(m^2-1)}  \frac{\partial \phi_w}{\partial u}\]		
		from Theorem \ref{thm:main}.	One checks that the inverse sum of adjoint torsions as in Theorem \ref{thm:main2} is numerically $-2q$ for generic $c\in\Cbb$.
	\end{remark}

	%	As a consequence, we obtain the following from \cite{porti1997torsion} (see ? for details).
	%	\begin{corollary} Suppose $\rho : \pi_1(K) \rightarrow \G$ is a $\gamma$-regular irreducible representation for a peripheral curve $\gamma$. Then we have 
	%		\[ \Tbb_\gamma(\rho)=\frac{1}{2(m^2-1)}\frac{w_{11}}{v_{11}\phi_v} \det \left(\frac{\partial (\phi_w,\gamma_{11})}{\partial (y,m)}\right).\]		
	%		Note that $\gamma_{11}$ is an eigenvalue function of $\gamma$.
	%	\end{corollary}
	%	\begin{proof} 
	%		a
	%	\end{proof}
	%
	%	\begin{remark}
	%		In bernard paper. the torsion 1-form $\tau$ 
	%		\[\tau = \frac{2(m-m^{-1}) v_{11} \phi_v}{w_{11}} \frac{dm}{\partial \phi_w/\partial y}\]
	%		If the Riley polynomial $\phi_w$ has more than one component, say $\phi_w = \phi_1 \cdots \phi_k$. Then any intersection point we have a pole		
	%	\end{remark}
	%	

	\subsection{A proof of Theorem \ref{thm:main}} \label{sec:proof1}
%	Recall that we have
%	\[\rho(w)= \rho(g_1^{\epsilon_1}g_2^{\epsilon_2}\cdots g_1^{\epsilon_{p-2}}g_2^{\epsilon_{p-1}})=\begin{pmatrix}
%	w_{11} & w_{12} \\ w_{21} & w_{22} 
%	\end{pmatrix}\] 
%	where $w_{ij}$ are polynomials in $m^{\pm1}$ and $u$. 
	Let $Y$ be the 2-dimensional cell complex corresponding to the presentation \eqref{eqn:prsn} of the knot group $\pi_1(M)$. 
	Recall that $Y$ has one 0-cell $p$, two 1-cells $g_1$, $g_2$, and one 2-cell $r$ and that we have the basis $\mathbf{c}^\ast$ of $C^\ast(Y;\g_\rho)$ as in the equation \eqref{eqn:cbasis}:
	$\mathbf{c}^0 = (p^1,p^2,p^3),\ \mathbf{c}^1=(g_1^1,g_1^2,g_1^3,g_2^1,g_2^2,g_2^3),\
	\mathbf{c}^2=(r^1,r^2,r^3).$
	
	Let $F_2$ be the free group generated by $g_1$ and $g_2$. We  consider a representation $\overline{\rho} : F_2 \rightarrow M_{2,2} (\Zbb[t_1^{\pm1},t_2^{\pm1},t_3])$ given by
	\begin{equation} \label{eqn:assign}
	\overline{\rho}(g_1) = \begin{pmatrix} t_1 & 1 \\ 0 & t_1^{-1}\end{pmatrix}, \ \overline{\rho}(g_2) =\begin{pmatrix} t_2 & 0 \\ -t_3 & t_2^{-1} \end{pmatrix}
	\end{equation} 
	and define one-parameter families $\rho^0_t,\rho^1_t, \rho^2_t,\rho^3_t : F_2 \rightarrow \G$ from $\overline{\rho}$ by letting
	\begin{equation*}
		\begin{array}{llll}
			\rho^0_t :&  t_1=t, & t_2=t,& t_3=u,\\
			\rho^1_t :&  t_1=t,& t_2=m,& t_3=u,\\
			\rho^2_t :&  t_1=m,& t_2=t,& t_3=u,\\
			\rho^3_t :&  t_1=m,& t_2=m,& t_3=t.		
		\end{array}
	\end{equation*}
	Clearly, $\rho^i_t$ coincides with $\rho$ at $t=m$ for $i=0,1,2$ and at $t=u$ for $i=3$. We thus have $A_{\rho_t^i}\in C^1(Y;\g_\rho)$ defined as in the equation \eqref{eqn:A} for $i=0,1,2,3$. 
	%	Note that we have $A_{\rho^0_t}=A_{\rho^1_t}+A_{\rho^2_t}$ from the fact that for any $g \in F_2$
	%	\[\left. \frac{d}{dt} \right |_{t=m} \rho_t^0(g)= \left. \frac{d}{dt} \right |_{t=m} \rho_t^1(g)+\left. \frac{d}{dt} \right |_{t=m} \rho_t^2(g).\]
	%	Indeed, 
	With respect to the basis $\mathbf{c}^1$, 
	\begin{equation} \label{eqn:A1}
	A_{\rho_t^0} = \begin{pmatrix} -1 \\ m^{-1} \\ 0 \\ 0 \\  m^{-1} \\ -u m^{-2} \end{pmatrix}, \
	A_{\rho_t^1} = \begin{pmatrix} -1 \\  m^{-1} \\ 0 \\ 0 \\ 0 \\0 \end{pmatrix}, \
	A_{\rho_t^2} = \begin{pmatrix} 0 \\0  \\ 0 \\ 0 \\ m^{-1} \\ -um^{-2} \end{pmatrix}, \ 
	A_{\rho_t^3} = \begin{pmatrix} 0 \\0  \\ 0 \\ 0 \\ 0\\ - m^{-1} \end{pmatrix}.
	\end{equation}
	Note that we have $A_{\rho^0_t}=A_{\rho^1_t}+A_{\rho^2_t}$ which is also clear from the definition.
	
	Recall that the Riley polynomial is equal to $\phi_w = w_{11} -(m-m^{-1})w_{12}$ where $w_{ij} \in \Zbb[m^{\pm1},u]$ is the $(i,j)$-entry of $\rho(w)$. Similarly, we let $\overline{\phi}_w :=\overline{w}_{11} -(t_1-t_1^{-1})\overline{w}_{12}$ where $\overline{w}_{ij} \in \Zbb[t_1^{\pm1},t_2^{\pm1},t_3]$ is the $(i,j)$-entry of $\overline{\rho}(w)$.
	
	\begin{lemma} \label{lem:adelta}We have
		\begin{align}
			(\delta^1 A_{\rho^0_t})(\widetilde{r})&=\frac{\partial \phi_w}{\partial m} M \label{eqn:azero},\\
			(\delta^1 A_{\rho^1_t})(\widetilde{r})&=\begin{pmatrix} m^{-1} & 0 \\ 0 & -m^{-1} \end{pmatrix}+  \alpha_1 M  +\begin{pmatrix} 0 & 0 \\ \ast & 0 \end{pmatrix} \label{eqn:aone},\\
			(\delta^1 A_{\rho^2_t})(\widetilde{r})&=\begin{pmatrix} -  m^{-1} & 0 \\ 0 & m^{-1} \end{pmatrix}+\alpha_2 M + \begin{pmatrix} 0 & 0 \\ \ast & 0 \end{pmatrix} \label{eqn:atwo},\\
			(\delta^1 A_{\rho^3_t})(\widetilde{r})&=\frac{\partial \phi_w}{\partial u} M \label{eqn:athree}
		\end{align}
	where 
	\begin{equation*}
		M=\begin{pmatrix}
			w_{11} u -w_{21}m^{-1}   & w_{11} m \\
			w_{21}u+w_{22}um^{-1}& w_{21}m
		\end{pmatrix},\ \alpha_1= \left.\frac{\partial \overline{\phi}_w}{\partial t_1}  \right|_{t_1=t_2=m,t_3=u},\ \alpha_2=\left.\frac{\partial \overline{\phi}_w}{\partial t_2}  \right|_{t_1=t_2=m,t_3=u}.
	\end{equation*}
	Note that we have $\alpha_1+\alpha_2 = \frac{\partial \phi_w}{\partial m}$ from the definition (also from the fact that $A_{\rho^0_t}=A_{\rho^1_t}+A_{\rho^2_t}$).
	\end{lemma}
	\begin{proof} 	A straightforward computation shows that
		\begin{align*}
			\overline{\rho}(r)  &=\overline{\rho}(w g_1 w^{-1}g_2^{-1}) \\
			& =\begin{pmatrix}
				t_1t_2^{-1} & 0 \\
				0 & t_1^{-1}t_2
			\end{pmatrix} + \overline{\phi}_w
			\begin{pmatrix}
				\overline{w}_{11} t_3 - \overline{w}_{21} t_2^{-1} & \overline{w}_{11} t_2 \\
				\overline{w}_{21} t_3 +\overline{w}_{22} t_3 t_2^{-1} & \overline{w}_{21} t_2 
			\end{pmatrix} \\
			& \quad \quad + \left( t_1^{-1}-t_2^{-1} \right)	
			\begin{pmatrix}
				0 & 0 \\
				t_3 & 0
			\end{pmatrix} - ( \overline{w}_{12} t_3+\overline{w}_{21})
			\begin{pmatrix}
				0 & 0 \\
				\overline{w}_{21}-(t_1-t_1^{-1})\overline{w}_{22}	 & 0
			\end{pmatrix}.
		\end{align*}
		Taking $t_1=t_2=m$ and $t_3=u$, the above equation reduces to the equation
		\begin{equation} \label{eqn:R}
		\rho(r) = \rho(w g_1 w^{-1} g_2^{-1})=I+\phi_w \begin{pmatrix}
		w_{11} u -w_{21}m^{-1}   & w_{11} m \\
		w_{21}u+w_{22}um^{-1}& w_{21}m
		\end{pmatrix}=I+\phi_w M.
		\end{equation}
		Note that the fact that $w_{12} u + w_{21}=0$ is used, see \cite{riley1984nonabelian}. We then obtain the equations \eqref{eqn:azero}--\eqref{eqn:athree} immediately  from the equation \eqref{eqn:r}. For instance,
		\begin{align*}
			(\delta^1 A_{\rho^0_t})(\widetilde{r}) &= \left. \frac{d}{dt} \right |_{t=m} \rho^0_t(r) = \frac{\partial \rho(r)}{\partial m} = \frac{\partial \phi_w}{\partial m} M + \phi_w \frac{\partial M}{\partial m} = \frac{\partial \phi_w}{\partial m} M,\\
		(\delta^1 A_{\rho^1_t})(\widetilde{r})&=\left. \frac{d}{dt} \right |_{t=m} \rho^1_t(r) =\left.\frac{\partial \overline{\rho}(r)}{\partial t_1} \right|_{t_1=t_2=m,t_3=u} \\
		&=\begin{pmatrix} m^{-1} & 0 \\ 0 & - m^{-1} \end{pmatrix}+\left. \frac{\partial \overline{\phi}_w}{\partial t_1} \right|_{t_1=t_2=m,t_3=u} M+ \begin{pmatrix} 0 & 0 \\ \ast & 0 \end{pmatrix}.
	\end{align*}
	We compute $(\delta^1 A_{\rho^2_t})(\widetilde{r})$ and $(\delta^1 A_{\rho^3_t})(\widetilde{r})$, similarly.

%		
%			 that 
%		\[	(\delta^1 A_{\rho^0_t})(\widetilde{r}) = \left. \frac{d}{dt} \right |_{t=m} \rho^0_t(r) = \frac{\partial \rho(r)}{\partial m}, \quad	(\delta^1 A_{\rho^3_t})(\widetilde{r}) = \left. \frac{d}{dt} \right |_{t=u} \rho_t^3(r)= \frac{\partial \rho(r)}{\partial u}.\]
%		%		Here entries of $\rho(r)$ are viewed as polynomials in $m^{\pm1}$ and $y$.
%		Since we have  $\rho(r) = I + \phi_w M$, we obtain the equations \eqref{eqn:azero} and \eqref{eqn:athree} as follow.
%		\[ \frac{\partial \rho(r)}{\partial m} = \, \quad \frac{\partial \rho(r)}{\partial u} = \frac{\partial \phi_w}{\partial u} M + \phi_w\frac{\partial M}{\partial u} = \frac{\partial \phi_w}{\partial u} M. \]
%		
%		vIt is proved in \cite{riley1984nonabelian} that $w_{12} u  + w_{21}=0$ and 
%		
%		Recall the equation \eqref{eqn:assign} that we have $\overline{\rho}(w) =\overline{\rho}(g_1^{\epsilon_1}\cdots g_2^{\epsilon_{p-1}})$ whose entries $\overline{w}_{ij}$ are polynomials in $t_1^{\pm1},t_2^{\pm1},t_3$. 
%	
%	
%		Thus we obtain the equations \eqref{eqn:aone} and \eqref{eqn:atwo} as follows.
%		and
%		\begin{align*}
%			\left. \frac{d}{dt} \right |_{t=m} \rho^2_t(r) &=\left.\frac{\partial \overline{\rho}(r)}{\partial t_2} \right|_{t_1=t_2=m,t_3=u} \\
%			&=\begin{pmatrix} -  \frac{1}{m} & 0 \\ 0 & \frac{1}{m} \end{pmatrix}+\left. \frac{\partial(\overline{w}_{11}-(t_1-t_1^{-1})\overline{w}_{12})}{\partial t_2} \cdot M\right|_{t_1=t_2=m,t_3=u} + \begin{pmatrix} 0 & 0 \\ \ast & 0 \end{pmatrix}.
%		\end{align*}
%		
	\end{proof}
	
	\begin{claim} \label{claim1} The following vector $B\in C^1(Y;\g_\rho)$ represents a generator of $H^1(Y;\g_\rho)$.
		\begin{equation} \label{eqn:claim}
		B=A_{\rho^0_t} - \frac{ \partial \phi_w/\partial m}{ \partial \phi_w/\partial u} A_{\rho^3_t} =
		A_{\rho^1_t} + A_{\rho^2_t} - \frac{ \partial \phi_w/\partial m}{ \partial \phi_w/\partial u} A_{\rho^3_t}.
		\end{equation}
	\end{claim}
	%	{\color{red} CHECK : $\partial \phi_w / \partial u \neq 0$.}
	\begin{proof} Lemma \ref{lem:adelta} immediately implies that $B \in \mathrm{Ker}\,\delta^1$.
		 The fact that $B \notin \mathrm{Im}\,\delta^0$ can be easily checked from the explicit expressions of $A_{\rho^i_t}$ (see the equation \eqref{eqn:A1}) and $\delta^0$ with respect to the basis $\mathbf{c}^\ast$
		\begin{equation}\label{eqn:d0}
		\delta^0 = \begin{pmatrix} m^2-1 & -2m & -1 \\ 0 & 0 & 	m^{-1} \\ 0 & 0 & m^{-2}-1 \\ m^2-1 & 0&0 \\  um	 & 0 & 0 \\ -u^2 & -2um^{-1} & m^{-2}-1 \end{pmatrix}.
		\end{equation}
	\end{proof}
	
	On the zero set of the Riley polynomial $\phi_w=w_{11}-(m-m^{-1})w_{12}$ with $m \neq \pm1$, we have $w_{11} \neq 0$ and $w_{12} \neq 0$. Otherwise, we have $w_{11}=w_{12}=0$ which contradicts to $\det \rho(w)=1$.  
	
	Recall Section \ref{sec:basic} that to compute the adjoint torsion we need to specify  a sequence $b^i$ of vectors in $C^i(Y;\g_\rho)$ such that $\delta^i(b^i)$ is a basis of $\mathrm{Im}\ \delta^i$ for $i=0,1$ and a basis $\mathbf{h}^i$ of $H^i(Y;\g_\rho)$ satisfying $F_i(\mathbf{h}^i)=1$ for $i=1,2$.
	\begin{claim} \label{claim2}
		The following choice of $b^i$ satisfies that $\delta^i(b^i)$ is a basis of $\mathrm{Im}\,\delta^i$. 	\begin{equation}\label{eqn:b}
		b^0 = \mathbf{c}^0, \quad b^1=(\beta_1  A_{\rho_t^1}+ \beta_2 A_{\rho_t^2},\ A_{\rho_t^3})
		\end{equation}
		for any $\beta_1 \neq \beta_2 \in \Cbb$.
	\end{claim}
	\begin{proof} 	It is clear that $\delta^0(b^0)$ is a basis of $\mathrm{Im} \,\delta^0$, since $\delta^0$ is injective.
		
		Since $\dim_\Cbb H^2(Y;\g_\rho)=1$, we have $\dim_\Cbb \mathrm{Im} \,\delta^1=2$.
		Suppose $\delta^1(b^1)$ fails to span $\mathrm{Im}\,\delta^1$. Then there is a linear combination $C$ of $\beta_1 A_{\rho_t^1} + \beta_2 A_{\rho_t^2}$ and $A_{\rho_t^3}$ such that $\delta^1(C)=0$. Recall that $w_{11} \neq 0$, $\frac{\partial \phi_w}{\partial u}\neq0$, and 
		\begin{equation} \label{eqn:r3}
		(\delta^1 A_{\rho_t^3})(\widetilde{r}) = \dfrac{\partial \phi_w}{\partial u}\begin{pmatrix}
			w_{11} u -w_{21}m^{-1}   & w_{11} m \\
			w_{21}u+w_{22}um^{-1}& w_{21}m
		\end{pmatrix}.
		\end{equation} It follows that $\delta^1(A_{\rho_t^3}) \neq 0$ and thus $C$ has a non-trivial coefficient for $\beta_1 A_{\rho^1_t}+\beta_2 A_{\rho^2_t}$.
		Since $A_{\rho_t^1},A_{\rho_t^2},$ and $A_{\rho_t^3}$ are linearly independent (see the equation \eqref{eqn:A1}), so are the vector $B$ in Claim \ref{claim1} and $C$. 
		Since $B$ generates $H^1(Y;\g_\rho)$, we have $C- k B
		\in \mathrm{Im}\,\delta^0$ for some $k \in \Cbb$. It follows that there exists $a_1 A_{\rho^1_t} + a_2 A_{\rho^2_t}+a_3 A_{\rho^3_t}\in \mathrm{Im}\,\delta^0$ for some $a_1,a_2,a_3 \in \Cbb$ with $a_1 \neq a_2$. However, one easily checks that such a vector can not exist from the explicit expressions \eqref{eqn:A1} and \eqref{eqn:d0}.
	\end{proof}
	
	We fix a meridian $\mu$ by $g_1$ (see Figure \ref{fig:schubert}~(left)) and choose $P \neq 0 \in \g$ invariant under the adjoint action for all peripheral curves as 
	\[ P= \begin{pmatrix} \frac{1}{2}(m-m^{-1})  & 1
	\\ 0 & -\frac{1}{2}(m-m^{-1}) \end{pmatrix}.\]	
	
	\begin{claim} \label{claim3}
		The following choice of $\widetilde{\mathbf{h}}^1 \in C^1(Y;\g_\rho)$ represents the basis $\mathbf{h}^1$ of $H^1(Y;\g_\rho)$ satisfying $F_1(\mathbf{h}^1)=1$.
		\begin{equation*}
			\widetilde{\mathbf{h}}^1 =\dfrac{1}{4(1-m^{-2})} \left(A_{\rho^0_t} - \frac{ \partial \phi_w/\partial m}{ \partial \phi_w/\partial u} A_{\rho^3_t}\right)
		\end{equation*}
	\end{claim}
	\begin{proof} The proof is immediately followed from Claim \ref{claim1} with  
		\[\langle P,B(\widetilde{g}_1) \rangle = 		\left\langle \begin{pmatrix} \frac{1}{2}(m-m^{-1}) & 1 \\ 0 & -\frac{1}{2}(m-m^{-1}) \end{pmatrix}, \begin{pmatrix} m^{-1} & -1 \\ 0 & - m^{-1} \end{pmatrix}  \right\rangle_\g = 4(1-m^{-2}).\]
	\end{proof}
	
	In the Schubert normal form, the relator $r = w g_1 w^{-1} g_2^{-1}$ is represented by a loop that travels along the knot diagram as in Figure \ref{fig:schubert}~(right).
	\begin{figure}[!h]
		\centering
		%% Creator: Inkscape 1.0 (4035a4fb49, 2020-05-01), www.inkscape.org
%% Accompanies image file 'schubert.pdf' (pdf, eps, ps)
\begingroup%
  \makeatletter%
  \providecommand\color[2][]{%
    \errmessage{(Inkscape) Color is used for the text in Inkscape, but the package 'color.sty' is not loaded}%
    \renewcommand\color[2][]{}%
  }%
  \providecommand\transparent[1]{%
    \errmessage{(Inkscape) Transparency is used (non-zero) for the text in Inkscape, but the package 'transparent.sty' is not loaded}%
    \renewcommand\transparent[1]{}%
  }%
  \providecommand\rotatebox[2]{#2}%
  \newcommand*\fsize{\dimexpr\f@size pt\relax}%
  \newcommand*\lineheight[1]{\fontsize{\fsize}{#1\fsize}\selectfont}%
  \ifx\svgwidth\undefined%
    \setlength{\unitlength}{351.87576174bp}%
    \ifx\svgscale\undefined%
      \relax%
    \else%
      \setlength{\unitlength}{\unitlength * \real{\svgscale}}%
    \fi%
  \else%
    \setlength{\unitlength}{\svgwidth}%
  \fi%
  \global\let\svgwidth\undefined%
  \global\let\svgscale\undefined%
  \makeatother%
  \begin{picture}(1,0.44568475)%
    \lineheight{1}%
    \setlength\tabcolsep{0pt}%
    \put(0,0){\includegraphics[width=\unitlength,page=1]{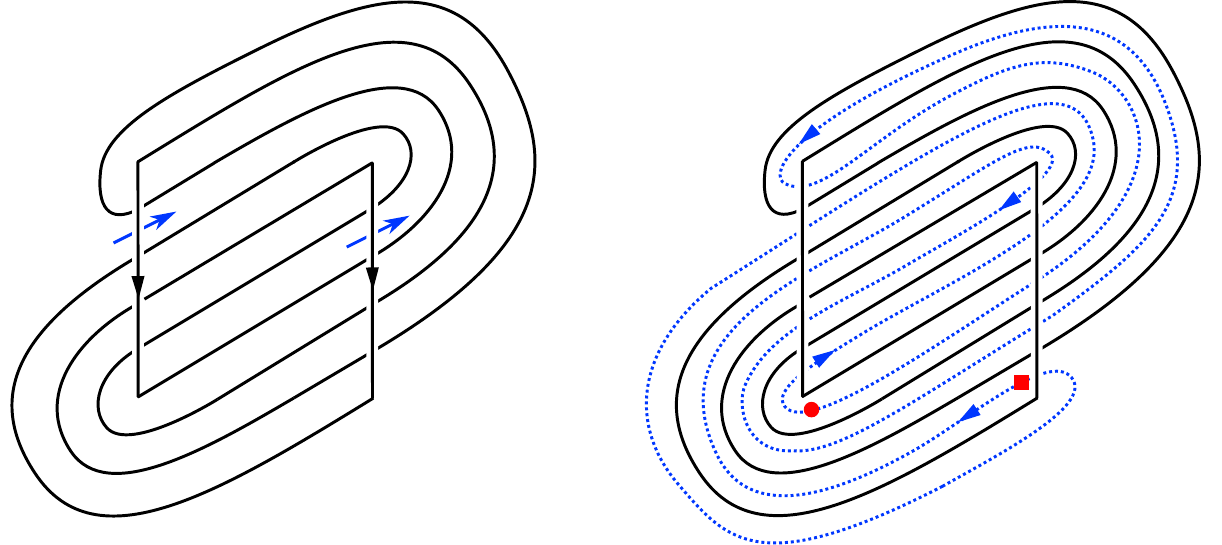}}%
    \put(0.06615889,0.23634118){\color[rgb]{0,0,0}\makebox(0,0)[lt]{\lineheight{1.25}\smash{\begin{tabular}[t]{l}$g_1$\end{tabular}}}}%
    \put(0.25946879,0.23158428){\color[rgb]{0,0,0}\makebox(0,0)[lt]{\lineheight{1.25}\smash{\begin{tabular}[t]{l}$g_2$\end{tabular}}}}%
  \end{picture}%
\endgroup%

		\caption{The generators $g_1,g_2$ and the loop representing the relator $r$ for $(p,q)=(5,3)$.}
		\label{fig:schubert}
	\end{figure}
	More precisely, if we follow the loop from the square-dot as in Figure \ref{fig:schubert}~(right), we obtain the word $w g_1 w^{-1} g_2^{-1}$. Similarly, if we start from the circle-dot as in Figure \ref{fig:schubert}~(right), we obtain the word $g_1 w^\dagger g_2^{-1} (w^\dagger)^{-1}$ along the loop, where $w^\dagger=g_2^{\epsilon_{1}} g_1^{\epsilon_2}\cdots g_1^{\epsilon_{p-1}}$ is the word obtained from $w$ by exchanging $g_1$ and $g_2$.
	Considering the base-point change from the square-dot to the circle-dot, we obtain
	\begin{equation} \label{eqn:v}
	v'^{-1} (w g_1 w^{-1} g_2^{-1})  v' = g_1 w^\dagger g_2^{-1} (w^\dagger)^{-1} \quad \textrm{where }
	v'= 
	\left \{	
	\begin{array}{ll}
	g_1^{\epsilon_1}g_2^{\epsilon_2}\cdots g_1^{\epsilon_{k-2}} g_2^{\epsilon_{k-1}} & \mathrm{if} \ \epsilon_k=1\\
	g_2g_1^{\epsilon_1}g_2^{\epsilon_2}\cdots g_2^{\epsilon_{k-1}}g_1^{\epsilon_k} & \mathrm{if} \ \epsilon_k=-1
	\end{array}
	\right.
	\end{equation}
	Here $0 < k < p$ is a unique odd integer satisfying $k q \equiv \pm 1$ in modulo $2p$.
	Since the loop $w^\dagger w$ is a longitude of $K$  (see \cite{riley1972parabolic, riley1984nonabelian}), the boundary torus $\partial M$ of $K$ is represented by $- v'^{-1} \cdot \widetilde{r} + w^\dagger \cdot \widetilde{r} \in C_2(\widetilde{Y};\Zbb)$ as in Figure \ref{fig:torus}.
	\begin{figure}[!h]
		\centering
		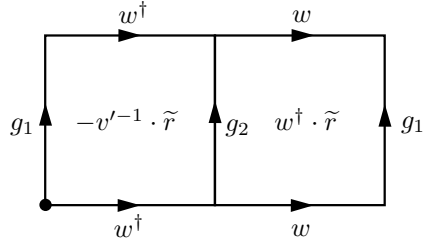
		\caption{The boundary torus $\partial M$.}
		\label{fig:torus}
	\end{figure}
	
	\begin{claim}\label{claim4}
		The following choice of $\widetilde{\mathbf{h}}^2 \in C^2(Y;\g_\rho)$ represents the basis $\mathbf{h}^2$ of $H^2(Y;\g_\rho)$ satisfying $F_2(\mathbf{h}^2)=1$.
		\begin{equation}\label{eqn:h2}
		\widetilde{\mathbf{h}}^2 = \begin{pmatrix} 0 & 0 & 
		-\dfrac{1}{4 v'_{11} \phi_{v'}}
		\end{pmatrix}^T
		\end{equation}
		Here we write an element of $\g$ as an element of $\Cbb^3$ with respect to the basis \eqref{eqn:basis}.
	\end{claim}
	\begin{proof} Recall Lemma \ref{lem:adelta} that $\alpha_1+\alpha_2 =\partial \phi_w/\partial m \neq 0$ and thus both $\alpha_1$ and $\alpha_2$ can not be zero.
	Without loss of generality, we assume $\alpha_1 \neq 0$. Considering the $e_1$- and $e_2$-coordinates of $(\delta^1 A_{\rho_t^0})(\widetilde{r})$ and $(\delta^1 A_{\rho_t^1})(\widetilde{r})$ given in Lemma \ref{lem:adelta}, it follows that any element $\eta\in H^2(Y;\g_\rho)$ has a representative  $\widetilde{\eta}\in C^2(Y;\g_\rho)$ such that $\widetilde{\eta}(\widetilde{r})=(0,0,c)^T$ for some $ c\in \Cbb$.
	On the other hand, 	one can easily check that  any vector $C \in \g$ satisfying $\langle P, C\rangle_{\g} =1$ is of the form 
	\begin{equation} \label{eqn:form}
	C=(0,0,1/4)^T +
	a(1,0,0)^T+
	b(0,\, -1,\, m-m^{-1})^T
	\end{equation}
	for some $a,b \in \Cbb$. 
	
	Recall that we showed (cf. Figure \ref{fig:torus}) that  an element $\widetilde{\eta} \in C^2(Y;\g_\rho)$ with $\widetilde{\eta}(\widetilde{r})=(0,0,c)^T$ satisfies 
	\begin{equation*}
		\widetilde{\eta}(\widetilde{\partial M})=\left(-\mathrm{Ad}_{\rho(v'^{-1})} +\mathrm{Ad}_{\rho(w^\dagger)}\right)(0,0,c)^T.
	\end{equation*}
	A simple computation shows that $	\widetilde{\eta}(\widetilde{\partial M})$ in the above equation has the form \eqref{eqn:form} if and only if
		\[c=\dfrac{1}{4 (w_{22}^\dagger(w_{22}^\dagger+(m-m^{-1}) w_{12}^\dagger) - v'_{11}(v'_{11}-(m-m^{-1})v'_{12}))}. \]
		This completes the proof, since $w_{22}^\dagger+(m-m^{-1}) w_{12}^\dagger$ coincides with the Riley polynomial $\phi_w$,  see \cite{riley1984nonabelian}.
	\end{proof}
	
	We now compute the determinant of the transition matrices.
	From Claim \ref{claim2} and \ref{claim3}, we have  $[b^0/\mathbf{c}^0]=1$ and
	\begin{align*}
		\left [	\dfrac{(\delta^0(b^0), \widetilde{\mathbf{h}}^1, b^1)}{\mathbf{c}^1} 
		\right ] &=
		\frac{1}{4(1-m^{-2})}
		\left [
		\dfrac{(\delta^0(\mathbf{c}^0),\, A_{\rho^0_t},\, \beta_1 A_{\rho^1_t} + \beta_2 A_{\rho^2_t},\, A_{\rho_t^3})}{\mathbf{c}^1} 
		\right ]\\
		&= 			\frac{1}{4(1-m^{-2})}\,
		\det 
		\begin{pmatrix}
			m^2-1 & -2m & -1 & -1 & -\beta_1&0\\
			0 & 0& m^{-1} & m^{-1} & \beta_1 m^{-1}&0\\
			0 & 0 & m^{-2}-1 & 0 & 0&0\\
			m^2-1 & 0 & 0 & 0 &0 & 0\\
			um & 0 & 0 & m^{-1} & \beta_2 m^{-1}& 0\\
			-u^2 & -2u m^{-1} & m^{-2}-1 & -u m^{-2} & -\beta_2 u m^{-2} & -m^{-1}
		\end{pmatrix} \\
		&=\frac{(1-m^{-2})(\beta_1-\beta_2)}{2}.
	\end{align*}
%	\begin{claim} There exists $\alpha \neq \beta \in \Cbb$ such that 	\begin{equation} \label{eqn:alphabeta}\alpha\frac{\partial(\overline{w}_{11}-(t_1-t_1^{-1})\overline{w}_{12})}{\partial t_1}+\beta\frac{\partial(\overline{w}_{11}-(t_1-t_1^{-1})\overline{w}_{12})}{\partial t_2}=0
%		\end{equation}
%		at $t_1=t_2=m$ and $t_3=u$.
%	\end{claim}
	We choose the constants  $\beta_1 \neq \beta_2$ such that $\alpha_1 \beta_1 + \alpha_2 \beta_2 =0$. Note that the fact $\alpha_1 + \alpha_2 = \frac{\partial \phi_w}{\partial m}\neq 0$ implies that $\beta_1 \neq \beta_2$. 
	Then Lemma \ref{lem:adelta} gives that
	\[(\delta^1(\beta_1 A_{\rho_t^1} +\beta_2 A_{\rho_t^2}))(\widetilde{r}) = \begin{pmatrix} \frac{\beta_1-\beta_2}{m} & 0 \\ \ast & -\frac{\beta_1-\beta_2}{m} \end{pmatrix}\]
	and from Claims \ref{claim2} and \ref{claim4} we obtain
	\begin{align*}
		\left[ \frac{\delta^1(b^1),\widetilde{\mathbf{h}}^2}{\mathbf{c}^2} \right]&=\left[ \frac{\delta^1(\beta_1 A_{\rho_t^1} +\beta_2 A_{\rho_t^2}),\,\delta^1(A_{\rho_t^3}),\,\widetilde{\mathbf{h}}^2}{\mathbf{c}^2} \right]\\
		&= \det
		\begin{pmatrix}  0 &  \frac{\partial \phi_w}{\partial u} w_{11} m &0\\
			\frac{\beta_1 -\beta_2}{m}  & \ast & 0 \\
			\ast & \ast & - \frac{1}{4v'_{11} \phi_{v'}} \end{pmatrix} \\
		&=\frac{ (\beta_1-\beta_2)w_{11} }{4v'_{11} \phi_{v'}}\frac{\partial \phi_w}{\partial u},
	\end{align*}
	where the second column is obtained from the equation \eqref{eqn:r3}. Since one easily checks that $v'_{11} \phi_{v'} = v_{11}\phi_v$ if $\epsilon_k=1$ and $v'_{11} \phi_{v'}=m^2 v_{11} \phi_v$ if $\epsilon_k=-1$ from the equation \eqref{eqn:v}, we obtain Theorem \ref{thm:main}
	\begin{equation}\label{eqn:fina} 	\Tbb_\mu(\chi_\rho) = \pm \left[\frac{b^0}{\mathbf{c}^0}\right] \left [	\dfrac{\delta^0(b^0), \widetilde{\mathbf{h}}^1, b^1}{\mathbf{c}^1} 
	\right ]^{-1}  		\left[ \frac{\delta^1(b^1),\widetilde{\mathbf{h}}^2}{\mathbf{c}^2} \right]=\pm \frac{ w_{11}}{2(1-m^{-2})v'_{11} \phi_{v'}} \frac{\partial \phi_w}{\partial u}=\pm\frac{m^{\epsilon_k+1} w_{11}}{2(m^2-1)v_{11} \phi_{v}} \frac{\partial \phi_w}{\partial u}.
	\end{equation}
	Since changing the character $\chi_\rho$, equivalently, changing the pair $(m,u)$, does not change any combinatorial data of the computation, the sign $\pm$ in the equation \eqref{eqn:fina} does not depend on the choice of $\chi_\rho$ (see \cite{turaev2001introduction} for details).
	\subsection{A proof of Theorem \ref{thm:main2}} \label{sec:proof2}

	Recall that the character variety $X_K$ of a two-bridge knot $K$ is given by 
	\[X_K = \{ (m,u) \in \Cbb^2 : \phi_w(m,u)=0,\ m\neq0,\ u \neq 0\}/_{\sim}\]
	where the quotient $\sim$ means that we identify $(m,u)$ and $(m^{-1},u)$. Since the trace function of $\mu$ is simply $m+m^{-1}$,   Theorem \ref{thm:main} implies that  for generic $c \in \Cbb$
	\begin{equation}\label{eqn:zero}
	\sum_{\chi_\rho \in \mathrm{tr}_\mu^{-1}(c)} \frac{1}{\mathbb{T}_\mu(\chi_\rho)} = \pm  \frac{d^{\epsilon_k+1}}{2(d^2-1)} \sum_{\phi_w(d,u)=0} \frac{v_{11}(d,u) \phi_v(d,u)}{w_{11}(d,u) \frac{\partial \phi_w}{\partial u}(d,u)}
	\end{equation}
	where $d \in \Cbb^\ast$ satisfies $d+d^{-1}=c$.
	Hereafter we regard every element in $\Cbb[m^{\pm1},u]$ as an element of $\Cbb[u]$ by letting $m=d$. For instance, $\phi_w = \phi_w(d,u) \in \Cbb[u]$; in particular, $\phi'_w$ means the derivative of $\phi_w$ with respect to the variable $u$.		
	
	Recall that we have $w=g_1^{\epsilon_1}g_2^{\epsilon_2} \cdots g_1^{\epsilon_{p-2}}g_2^{\epsilon_{p-1}}$ and $v=g_1^{\epsilon_1} g_2^{\epsilon_2} \cdots g_1^{\epsilon_{k-2}}g_2^{\epsilon_{k-1}}$ where $0 < k < p$ is a unique odd integer satisfying $k q \equiv \pm 1$ in modulo $2p$. 
	Let $y = g_1^{\epsilon_k}g_2^{\epsilon_k+1}\cdots g_1^{\epsilon_{p-2}}g_2^{\epsilon_{p-1}}$ so that $w=vy$. From $v= wy^{-1}$ we have $v_{11} = w_{11} y_{22}- w_{12}y_{21}$ and thus $v_{11}=-w_{11}(y_{21}-(d-d^{-1})y_{22})/(d-d^{-1})$ on the zero set of the Riley polynomial $\phi_w= w_{11}-(d-d^{-1})w_{12}$. Therefore,
	\begin{equation} \label{eqn:first}
		\sum_{\phi_w=0} \frac{v_{11} \phi_v}{w_{11} \phi'_w}=-\frac{1}{d-d^{-1}} \sum_{\phi_w=0} \frac{(y_{21}-(d-d^{-1})y_{22})\phi_v}{ \phi_w'}.
	\end{equation}
	We will claim that 
	$(y_{21}-(d-d^{-1})y_{22}) \phi_v= \alpha \phi_w + h$ for some $\alpha \in \Cbb$ and $h \in \Cbb[u]$ with $\deg h \leq \deg \phi_w-2$. Then we obtain Theorem \ref{thm:main2} by applying the Euler-Jacobi theorem, saying that for any $f$ and $g \in \Cbb[u]$ 
	\[\sum_{f=0} \frac{g(u)}{f'(u)}=0\]
	whenever  $f$ has a non-zero constant term with no double root and $g $ satisfies $\deg g \leq \deg f -2$. We refer to \cite[Chapter 5]{griffiths1978principles} for details.
	Note that the Riley polynomial $\phi_w$ has a non-zero constant term with no double root for generic $d\in\Cbb^\ast$, see \cite[Lemma 3]{riley1984nonabelian}.

%		For generic $d \in \Cbb^\ast$ we may assume that $w_{11} \phi_w$ has a non-zero constant term with no double root, see \cite[Theorem 3]{riley1972parabolic} and .
%	Also, from Lemma \ref{lem:key} below we have
%	$\deg w_{11} = \deg \phi_{w} = \frac{p-1}{2}$ and $\deg v_{11} = \deg \phi_{v} = \frac{k-1}{2}$. It follows that 
%	$\deg (w_{11} \phi_w) - \deg (v_{11}\phi_v) = p - k \geq 2$.
%	Applying the Euler-Jacobi theorem, we obtain 
%	\begin{equation} \label{eqn:first}
%	\sum_{w_{11}\phi_w=0} \frac{v_{11} \phi_v}{(w_{11} \phi_w)'}=\sum_{w_{11}=0} \frac{v_{11} \phi_v}{w'_{11} \phi_w}+\sum_{\phi_w=0} \frac{v_{11} \phi_v}{w_{11} \phi'_w}=0.
%	\end{equation}
	
	For $f$ and $g \in \Cbb[u]$ we write $f  = g + o(u^n)$ if $f-g$ has degree less than $n\geq0$; $f=g+o(u^0)$ means $f=g$.
	
	\begin{lemma} \label{lem:key} Let  $(f_1,\cdots,f_{2n})$ be a sequence of $f_j \in \{\pm1\}$ of even length $2n>0$ and let $h_j := g_1^{f_1} g_2^{f_2}\cdots g_1^{f_{2j-1}} g_2^{f_{2j}}$ for $1 \leq j \leq n$. Then $h_j$ satisfies
		\begin{equation} \label{eqn:induc}
		\rho(h_j)=\pm\begin{pmatrix}
		u^{j} + A_j u^{j-1}+o(u^{j-1})	&  \ -f_{2j} d^{-f_{2j}} \left( u ^{j-1} +o(u^{j-1})\right)\\[3pt]
		f_1 d ^{-f_1} \left(u^j + B_j u^{j-1} + o(u^{j-1})\right) & \ -f_1 f_{2j} d^{-f_1-f_{2j}} \left(u^{j-1}  + o(u^{j-1})\right)
		\end{pmatrix}
		\end{equation}
		for some $A_j$ and $B_j \in \Cbb$ where the value $A_j - B_j$ does not depend on $j$.
	\end{lemma}
	\begin{proof} 	Recall that we have		\begin{equation*} 
			\rho(g_1) = \begin{pmatrix} d & 1 \\ 0 & d^{-1} \end{pmatrix}, \ \rho(g_2) = \begin{pmatrix} d & 0 \\ -u & d^{-1} \end{pmatrix}.
		\end{equation*}
		A simple computation shows that 
		\[\rho\left(g_1^{f_i} g_2^{f_j}\right)=\pm 
		\begin{pmatrix}
		u - f_i f_j d^{f_i+f_j} & -f_j d^{-f_j} \\[3pt]
		f_i d^{-f_i} u & - f_i f_j d^{-f_i-f_j}
		\end{pmatrix}
		\] for all $f_i, f_j \in \{\pm1\}$. With the above equation, a routine induction proves that $\rho(h_j)$ has the form \eqref{eqn:induc} with relations
		\begin{align*}
			A_{j+1}&=A_j-f_{2j}f_{2j+1} d^{-f_{2j}-f_{2j+1}} -f_{2j+1}f_{2j+2} d^{f_{2j+1}+f_{2j+2}}, \\
			B_{j+1}&=B_j-f_{2j}f_{2j+1} d^{-f_{2j}-f_{2j+1}} -f_{2j+1}f_{2j+2} d^{f_{2j+1}+f_{2j+2}}. 
		\end{align*}
	In particular, we have $A_{j+1}-B_{j+1}=A_j - B_j$.
	\end{proof}

	From Lemma \ref{lem:key} we have
	\begin{align}
		\rho(v)& =\pm\begin{pmatrix}
			u^{\frac{k-1}{2}} + V_{11} u^{\frac{k-3}{2}}+o(u^{\frac{k-3}{2}})	&  \ -\epsilon_{k-1} d^{-\epsilon_{k-1}} \left(u ^{\frac{k-3}{2}} + o(u^{\frac{k-3}{2}})\right)\\[3pt]
			\epsilon_1 d ^{-\epsilon_1} \left(u^\frac{k-1}{2} + V_{21} u^{\frac{k-3}{2}} +o(u^{\frac{k-3}{2}})\right) & \ -\epsilon_1 \epsilon_{k-1} d^{-\epsilon_1-\epsilon_{k-1}} \left (u^{\frac{k-3}{2}} +  o(u^{\frac{k-3}{2}})\right) 
		\end{pmatrix} \label{eqn:vm}\\
		\rho(y)& =\pm\begin{pmatrix}
			u^{\frac{p-k}{2}} + Y_{11} u^{\frac{p-k-2}{2}}+ o(u^{\frac{p-k-2}{2}})	&  \ -\epsilon_{p-1} d^{-\epsilon_{p-1}} \left( u ^{\frac{p-k-2}{2}} + o(u^{\frac{p-k-2}{2}})\right)\\[3pt]
			\epsilon_k d^{-\epsilon_k} \left(u^\frac{p-k}{2} + Y_{21} u^{\frac{p-k-2}{2}} +o(u^{\frac{p-k-2}{2}})\right) & \ -\epsilon_k \epsilon_{p-1} d^{-\epsilon_k-\epsilon_{p-1}} \left(u^{\frac{p-k-2}{2}} +  o(u^{\frac{p-k-2}{2}})\right)
		\end{pmatrix} \label{eqn:ym}\\
		\rho(w)& =\pm\begin{pmatrix}
			u^{\frac{p-1}{2}} + W_{11} u^{\frac{p-3}{2}}+o(u^{\frac{p-3}{2}})	&  \ -\epsilon_{p-1} d^{-\epsilon_{p-1}} \left( u ^{\frac{p-3}{2}} + o(u^{\frac{p-3}{2}})\right)\\[3pt]
			\epsilon_1 d^{-\epsilon_1} \left(u^\frac{p-1}{2} + W_{21} u^{\frac{p-3}{2}} +o(u^{\frac{p-3}{2}})\right) & \ -\epsilon_1 \epsilon_{p-1} d^{-\epsilon_1-\epsilon_{p-1}} \left(u^{\frac{p-3}{2}} +  o(u^{\frac{p-3}{2}})\right) 
		\end{pmatrix} \label{eqn:wm}
	\end{align}
	for some $V_{ij},Y_{ij},W_{ij}\in\Cbb$. Note that $k\neq 1$ (and thus $k\geq3$) since $q \neq \pm1$ and that 
			\begin{equation}\label{eqn:phiw} \phi_w = \pm \left(u^\frac{p-1}{2} + (W_{11} +\epsilon_{p-1} d^{-\epsilon_{p-1}} (d-d^{-1})) u^\frac{p-3}{2} + o(u^\frac{p-3}{2})\right).
			\end{equation}
		Computing the first and second leading terms of $(y_{21}-(d-d^{-1})y_{22})\phi_v$,
	\begin{align*}
		 & (y_{21}-(d-d^{-1})y_{22})\phi_v\\
		&=	 \pm\epsilon_k d ^{-\epsilon_k} \left(u^\frac{p-k}{2} + (Y_{21}+\epsilon_{p-1}d^{-\epsilon_{p-1}}(d-d^{-1})) u^{\frac{p-k-2}{2}} + o(u^{\frac{p-k-2}{2}})\right)\\
		&\quad \cdot \left(u^{\frac{k-1}{2}} + (V_{11}+\epsilon_{k-1}d^{-\epsilon_{k-1}}(d-d^{-1})) u^{\frac{k-3}{2}}+o(u^\frac{k-3}{2})\right) \\
		&= \pm\epsilon_k d^{-\epsilon_k}  \left(u^{\frac{p-1}{2}}+(V_{11}+Y_{21} +\epsilon_{k-1}d^{-\epsilon_{k-1}}(d-d^{-1})+\epsilon_{p-1}d^{-\epsilon_{p-1}}(d-d^{-1}))u^\frac{p-3}{2} + o(u^\frac{p-3}{2})\right).
	\end{align*}
	\begin{claim}\label{claim5} 
		$W_{11}=V_{11}+Y_{21} +\epsilon_{k-1}d^{-\epsilon_{k-1}}(d-d^{-1})$.
	\end{claim}
	\begin{proof} Comparing the coefficient of the second leading terms in $w_{11} = v_{11} y_{11} +v_{12} y_{21}$, we obtain
		\begin{equation*}  W_{11}=V_{11}+Y_{11}-\epsilon_{k-1}\epsilon_k d^{-\epsilon_{k-1}-\epsilon_k}.
		\end{equation*}
		Thus it is enough to show that $Y_{11}-Y_{21} = \epsilon_{k-1}\epsilon_k d^{-\epsilon_{k-1}-\epsilon_k}+\epsilon_{k-1}d^{-\epsilon_{k-1}}(d-d^{-1})$.
	On the other hand, from Lemma \ref{lem:key} and
		\[\rho\left(g_1^{\epsilon_k} g_2^{\epsilon_{k+1}}\right)=\pm 
		\begin{pmatrix}
		u - \epsilon_k \epsilon_{k+1} d^{\epsilon_k+\epsilon_{k+1}} & -\epsilon_{k+1} d^{-\epsilon_{k+1}} \\
		\epsilon_k d^{-\epsilon_k} u & - \epsilon_k \epsilon_{k+1} d^{-\epsilon_k-\epsilon_{k+1}}
		\end{pmatrix},
		\]
		we obtain $Y_{11}-Y_{21} = -\epsilon_k \epsilon_{k+1} d^{\epsilon_k+\epsilon_{k+1}}$. It follows that we need to check the equation
		\begin{equation}\label{eqn:ek}\epsilon_{k-1}\epsilon_k d^{-\epsilon_{k-1}-\epsilon_k}+\epsilon_{k-1}d^{-\epsilon_{k-1}}(d-d^{-1})+ \epsilon_k \epsilon_{k+1} d^{\epsilon_k+\epsilon_{k+1}}=0.
		\end{equation}
		Since the integer $k$ satisfies $ k q \equiv \pm1 \ (\mathrm{mod} \ 2p)$ with $-p < q <p$,  the sequence $(\epsilon_{k-1},\epsilon_k,\epsilon_{k+1})$ can not be one of  $(1,1,1)$, $(-1,-1,-1)$, $(1,-1,1)$, and $(-1,1-1)$. It leaves four possible cases, $(\epsilon_{k-1},\epsilon_{k},\epsilon_{k+1})=(1,1,-1)$, $(1,-1,-1)$, $(-1,1,1)$, and $(-1,-1,1)$. One can easily check that the equation \eqref{eqn:ek} holds for these four cases.
	\end{proof}
	
	This completes the proof, since Claim \ref{claim5} implies that $(y_{21}-(d-d^{-1})y_{22})\phi_v = \pm \epsilon_k d^{-\epsilon_k} \phi_w + h$ for some $h \in \Cbb[u]$ with $\deg h \leq \frac{p-1}{2}-2$.
%	\subsection{Examples}
%		Let us consider the figure-eight knot $(p,q)=(5,3)$. Since $3 \cdot 3 \equiv \pm 1$ in modulo $2 \cdot 5$, we have
%	\[ \Tbb_\mu(\chi_\rho) = \pm\frac{\left(-2 m^2 u+m^4-3 m^2+1\right) \left(-u^2+(m^2-2) u -1\right)}{2 m^2(m^2-1)(u+1) \left(-u+m^2-2\right)} \]
	\bibliographystyle{alpha}
	\bibliography{biblog}
\end{document}